\documentclass[preprint,12pt]{elsarticle}

\usepackage{amssymb}
\usepackage{amsmath}
\usepackage{bm}
\usepackage{amsthm}
\usepackage{comment}

\newcommand{\adots}{
  \mathinner{\mkern1mu\raise1pt\hbox{.}\mkern2mu\raise4pt\hbox{.}
  \mkern2mu\raise7pt\vbox{\kern7pt\hbox{.}}\mkern1mu}}

\newcommand{\Sign}{\mathop{\mathrm{Sign}}}
\renewcommand{\Re}{\mathop{\mathrm{Re}}}
\renewcommand{\Im}{\mathop{\mathrm{Im}}}
\newcommand{\tr}{\mathop{\mathrm{tr}}}
\def\Neg{\mathop{\mathrm{neg}}}
\def\Pos{\mathop{\mathrm{pos}}}
\def\Ker{\mathop{\mathrm{ker}}}
\def\Im{\mathop{\mathrm{im}}}
\newcommand{\Imag}{\mathop{\mathrm{Im}}}

\usepackage{bm}
\usepackage{mathtools}
\DeclareMathAlphabet{\matheul}{U}{eus}{m}{n}

\theoremstyle{plain}
\newtheorem{thm}{Theorem}[section]
\newtheorem{lem}[thm]{Lemma}
\newtheorem{cor}[thm]{Corollary}
\newtheorem{pro}[thm]{Proposition}
\theoremstyle{definition}
\newtheorem{defn}[thm]{Definition}

\journal{Linear Algebra and its Applications}

\begin{document}

\begin{frontmatter}



\title{Factorizations into Normal Matrices \\ in Indefinite Inner Product Spaces}


\author{Xuefang Sui\fnref{e1}}
\author{Paolo Gondolo\fnref{e2}}

\fntext[e1]{{\it Email address:} xuefang.sui@utah.edu}
\fntext[e2]{{\it Email address:} paolo.gondolo@utah.edu}

\address{Department of Physics and Astronomy, University of Utah, 115 South 1400 East \#201, Salt Lake City, UT 84112-0830}

\begin{abstract}
We show that any nonsingular (real or complex) square matrix can be factorized into a product of at most three normal matrices, one of which is unitary,  another selfadjoint with eigenvalues in the open right half-plane, and the third one is normal involutory with a neutral negative eigenspace (we call the latter matrices normal neutral involutory).  Here the words normal, unitary, selfadjoint  and neutral  are understood with respect to an indefinite inner product. 
\end{abstract}  


\end{frontmatter}



\section{Introduction} 

We consider two kinds of indefinite inner products: a complex Hermitian inner product  and a real symmetric inner product.   Let $\mathbb{K}$ denote either the field of complex numbers $\mathbb{K} = \mathbb{C}$ or the field of real numbers $\mathbb{K} = \mathbb{R}$.    Let the nonsingular matrix $H \in \mathbb{K}^{n \times n}$ be either a complex Hermitian matrix $H\in \mathbb{C}^{n\times n}$ defining a complex Hermitian inner product $[x,y]_H=\bar{x}^THy$ for $x,y \in \mathbb{C}^n$, or a real symmetric matrix $H\in\mathbb{R}^{n\times n}$ defining a real symmetric inner product $[x,y]_H=x^THy$ for $x,y\in \mathbb{R}^{n}$. Here $\bar{x}^T$ indicates the complex conjugate transpose of $x$ and $x^T$ indicates the transpose of $x$. 

In a Euclidean  space,  $H \in \mathbb{K}^{n\times n}$ is the identity matrix.   In general, if the integers $p$ and $q$ denote  the numbers of positive and negative eigenvalues of $H \in \mathbb{K}^{n \times n}$, respectively,  one defines the inertia and the signature of the nondegenerate indefinite inner product $[x,y]_H$ as the pair $(p, q)$ and the integer $p-q$, respectively.

Extending the  definitions in a Euclidean space to an indefinite inner product space,  the $H$-adjoint $A^{[H]}$ of a matrix $A\in \mathbb{K}^{n \times n}$ is defined as $[Ax,y]_H=[x,A^{[H]}y]_H$ for all $x,y \in \mathbb{K}^n$,  that is 
\begin{align}
A^{[H]}=H^{-1}A^\dagger H
\end{align}
with
\begin{align}
A^\dagger=\begin{cases}
\bar{A}^T,   \quad \text{for the complex Hermitian inner product}, \\
A^T,  \quad \text{for the real symmetric inner product}.
\end{cases}
\end{align}
A matrix $A\in \mathbb{K}^{n\times n}$ is called selfadjoint (Hermitian for $\mathbb{K=C}$,  and symmetric for $\mathbb{K=R}$) if $A^\dagger=A$. 
A matrix $L \in  \mathbb{K}^{n\times n}$ is called $H$-unitary (for $\mathbb{K}=\mathbb{R}$, also called $H$-orthogonal) if $L^{[H]}L =LL^{[H]}=I_n$ or $L^\dagger H L=H$.   A matrix $S \in  \mathbb{K}^{n\times n}$ is  called  $H$-selfadjoint ($H$-Hermitian for $\mathbb{K}=\mathbb{C}$, and  $H$-symmetric for $\mathbb{K}=\mathbb{R}$) if $S^{[H]}=S$ or $S^\dagger H=HS$. Finally a matrix $A \in  \mathbb{K}^{n\times n}$ is called $H$-normal if it commutes with its $H$-adjoint, i.e., $A A^{[H]} = A^{[H]} A$.

There are several kinds of matrix factorizations in a Euclidean space.  Inspired by the polar form of nonzero complex numbers, a square matrix admits a polar decomposition into a product of a unitary matrix and a positive-semidefinite selfadjoint matrix.  The individual matrix factors in a polar decomposition are normal matrices.   Here,  as usual,  a matrix $A$ is normal if $A^\dagger A=A A^\dagger$.

In this paper,  we present a way to decompose any nonsingular matrix into matrix factors that are normal with respect to a predefined inner product (i.e., if the matrix $H$ defines the inner product, a matrix $A$ is $H$-normal if $A^{[H]} A=A A^{[H]}$). Our $H$-normal factorization is close to the polar decomposition in a Euclidean space.  We start by mentioning the classical polar decomposition and studies of polar decompositions in  indefinite inner product spaces, and then we present our results. 

To simplify the language,  we have extended the definition of  positive-definite matrices to matrices with nonreal eigenvalues as follows. 
\begin{defn}($r$-positive-definite)
\label{defn:r-positive-definite}
A matrix is $r$-positive-definite if all of its eigenvalues have positive real part,  or equivalently if all of its eigenvalues lie in the open right half-plane. 
\end{defn} 

We use  $A^{1/2}$  to denote the principal square root of a square matrix $A \in \mathbb{C}^{n\times n}$ defined as follows.
\begin{defn}
(Principal square root; Thm.~1.29 in~\cite{Higham2008}) For a nonsingular matrix $A \in \mathbb{C}^{n\times n}$ with no negative real eigenvalues, the principal square root $A^{1/2}$ of $A$ is the unique $r$-positive-definite solution $S$ of $S^2=A$. 
\end{defn}

We  have defined positive eigenspace, negative eigenspace and  nonreal eigenspace of a matrix $A \in \mathbb{C}^{n\times n}$ through the Jordan decomposition as follows.
\begin{defn}The positive (negative) eigenspace of  a matrix $A \in \mathbb{C}^{n\times n}$  is defined as the subspace spanned by a set of  generalized eigenvectors belonging to all the  positive (negative) real eigenvalues of $A$.   The nonreal eigenspace of  $A$  is defined as the subspace spanned by a set of  generalized eigenvectors belonging to all the   nonreal  eigenvalues of $A$.
\end{defn}

In the following, the concept of hyperbolic subspace will play an important role. We use the definition of hyperbolic subspace that appears in the context of quadratic forms and Witt's decomposition theorem as follows.
\begin{defn} (Hyperbolic subspace) A hyperbolic subspace is defined as a nondegenerate subspace with signature zero.  
\end{defn}
A hyperbolic subspace has necessarily an even number of dimensions.   A hyperbolic subspace of dimension $2m$ has a basis $(u_1,\dots, u_m, v_1,\dots, v_m)$ such that $[u_i,u_j]_H=[v_i,v_j]_H=0$ and $[u_i,v_j]_H=\delta_{ij}$, where $\delta_{ij}$ is Kronecker delta and $i,j=1,\hdots, m$. The two subspaces spanned by $(u_1,\dots, u_m)$ and  $(v_1,\dots, v_m)$, respectively, are neutral subspaces.  See, e.g.,~\cite{Szymiczek1997,Vaz2016}.   Let $w_i=\frac{1}{\sqrt{2}}(u_i+v_i)$ and $z_i=\frac{1}{\sqrt{2}}(u_i-v_i)$, then $(w_1, \dots,w_m,  z_1,\dots, z_m)$ is another basis that $[w_i,w_j]_H=\delta_{ij}$, $[z_i,z_j]_H=-\delta_{ij}$ and $[w_i, z_j]_H=0$. It is clear that the inertia of this hyperbolic subspace is $(m,m)$, i.e., the signature is zero.  Moreover, it is an orthogonal sum of $m$ hyperbolic planes, which are hyperbolic subspaces of dimension 2.

For nonsingular matrices, the classical concept of polar decomposition in a Euclidean space is expressed by the following statement.  
Any nonsingular square matrix $F \in \mathbb{K}^{n\times n} $ has unique right and left polar decompositions 
\begin{align}
\label{eq:F=OS}
F = U S =S^\prime U,
\end{align}
where the matrix $ U \in  \mathbb{K}^{n\times n} $ is unitary  and  the matrices $ S, S^\prime \in \mathbb{K}^{n\times n} $ are selfadjoint positive-definite (i.e., all of their eigenvalues are real and positive). The matrices $S$, $S'$, and $U$ are given by $S = ( F^\dagger F )^{1/2}$, $S^\prime = (F F^\dagger)^{1/2}$, and $U = F S^{-1} = S^{\prime-1} F$.  

It is a longstanding and interesting question to generalize the classical polar decomposition in a Euclidean space to an indefinite inner product space.  Generalized polar decomposition,  $H$-polar decomposition and semidefinite $H$-polar decomposition are defined and studied.  All of these definitions allow for singular matrices.  Also a generalized polar decomposition is defined in an indefinite scalar product space, where the product matrix is not necessarily selfadjoint.   Considering in this paper we study decompositions for  nonsingular matrices in indefinite inner product spaces,   here we quote their results only in the same situation. 

Let $H\in \mathbb{K}^{n\times n}$ be a selfadjoint matrix and $F\in \mathbb{K}^{n\times n}$ be a nonsingular square matrix.   Write $F\in\mathbb{K}^{n\times n}$  into a product of factors $L\in\mathbb{K}^{n\times n}$ and $S\in\mathbb{K}^{n\times n}$ as 
\begin{align}
\label{eq:Hpolar}
F = L S. 
\end{align}

In the form of  Equation~(\ref{eq:Hpolar}), $F\in\mathbb{K}^{n\times n}$ has a generalized polar decomposition  if  $L\in\mathbb{K}^{n\times n}$ is $H$-unitary and $S\in\mathbb{K}^{n\times n}$ is $r$-positive-definite $H$-selfadjoint (see~\cite{Cardoso2002,Higham2004,Higham2005,Higham2010,Mackey2006}). 
Necessary and sufficient conditions  for the existence of generalized polar decomposition  are given  in\cite{Higham2005,Higham2010} as follows.  A nonsingular matrix $F$ has a generalized polar decomposition if and only if   $F^{[H]} F$ has no negative real eigenvalues. When such a factorization exists, it is unique. 

In the form of  Equation~(\ref{eq:Hpolar}), $F\in\mathbb{K}^{n\times n}$ has an $H$-polar decomposition if $L\in\mathbb{K}^{n\times n}$ is $H$-unitary and $S\in\mathbb{K}^{n\times n}$ is  $H$-selfadjoint. In this case $S$ is not necessarily $r$-positive-definite. 
Necessary and sufficient conditions for the existence of an $H$-polar decomposition are given  in\cite{Bolshakov1995,Bolshakov1996,Bolshakov19961997,Bolshakov1997,Bolshakov19971997,Kintzel2005,Mehl2005} as follows. A nonsingular matrix $F$ has  an $H$-polar decomposition  if and only if either $F^{[H]}F$ has no negative  real  eigenvalues or the negative-real-eigenvalue  Jordan blocks  in the canonical form of $ ( F^{[H]} F, H ) $ come in pairs of opposite sign characteristic,  that is,  by Theorem 4.4 in~\cite{Bolshakov1997},  the part of the canonical form $(J,K)$ of $(F^{[H]}F, H)$ corresponding to  the negative eigenvalues $\lambda_l$ of $F^{[H]}F$ is 
\begin{align}
\label{eq:polar_condition}
\left(\bigoplus_l  \begin{pmatrix} J_{s_l}(\lambda_l)&\\&J_{s_l}(\lambda_l)\end{pmatrix}, \quad   \bigoplus_l \begin{pmatrix}Z_{s_l}& \\& -Z_{s_l} \end{pmatrix} \right).
\end{align}

In the  form of  Equation~(\ref{eq:Hpolar}),    $F\in\mathbb{K}^{n\times n}$ has a semidefinite $H$-polar decomposition if $L\in\mathbb{K}^{n\times n}$ is $H$-unitary and $S\in\mathbb{K}^{n\times n}$ is  $H$-selfadjoint and $H$-nonnegative, i.e., $HS$ is selfadjoint and positive-semidefinite (in a nonsingular case, $HS$ is positive-definite).  Necessary and sufficient conditions for the existence of a semidefinite $H$-polar decomposition are given in \cite{Bolshakov19971997} as follows.  A nonsingular matrix  $F$ has a semidefinite $H$-polar decomposition if and only if $F^{[H]}F$ has only positive real eigenvalues and is diagonalizable.  A semidefinite $H$-polar decomposition  is a particular case of an $H$-polar decomposition.

As seen from the statements above,  if $F$ has a semidefinite polar decomposition,  then $F$ has a generalized polar decomposition.  If  $F$ has a generalized polar decomposition, then $F$ has an $H$-polar decomposition.  

    In our previous work~\cite{Sui2015},   we found a unique indefinite polar decomposition in an indefinite inner product space $\mathbb{K}^n$, which is close to the generalized polar decomposition that the $H$-selfadjoint factor is $r$-positive-definite.    By introducing a proper sign function,  a square root of a negative eigenvalue is avoided.   The decompositions apply to all the nonsingular matrices. 
The result in~\cite{Sui2015} is more general which is studied  for both bilinear and sesquilinear forms in indefinite scalar product spaces.     Here we quote the results only in  indefinite inner product spaces.
Any nonsingular matrix $F \in \mathbb{K}^{n\times n}$ can be uniquely decomposed as 
\begin{align}
\label{eq:F=WS=SW}
F=WS = S'W,
\end{align}
where, with $\Sigma=\Sign (F^{[H]}F)$ and $\Sigma^\prime=\Sign (FF^{[H]})$, $W\in\mathbb{K}^{n\times n}$ is $(H, H\Sigma)$-unitary and $(H\Sigma^\prime, H)$-unitary,  $S\in \mathbb{K}^{n \times n}$ is $r$-positive-definite $H$-selfadjoint and $H\Sigma$-selfadjoint,  and $S^\prime \in \mathbb{K}^{n\times n}$ is  $r$-positive-definite $H$-selfadjoint and $H\Sigma^\prime$-selfadjoint.  Both right and left decompositions are unique.  $S$ is given by $S=(\Sigma F^{[H]}F)^{1/2}$ and $S^\prime$ is given by $S^\prime=(\Sigma^\prime FF^{[H]})^{1/2}$.  Here $\Sign$ is a sign function defined in Section~\ref{sec:sign_function}.

In this paper,  we  first show that any square matrix $W\in\mathbb{K}^{n\times n}$ such that $W^{[H]} W=\Phi$, where $\Phi$ is $H$-selfadjoint involutory with a hyperbolic negative eigenspace,  can be factorized into a product
\begin{align}
W = L X,
\label{eq:introlx}
\end{align}
where $L\in \mathbb{K}^{n\times n}$ is $H$-unitary and $X\in \mathbb{K}^{n\times n}$ is $H$-normal $H$-neutral involutory. We call a matrix $H$-neutral involutory  if  it is involutory with an $H$-neutral negative eigenspace (see Definition~\ref{defn:neutral_involution}).    Properties of $H$-normal $H$-neutral involutory matrices are presented in Section~\ref{sec:neutral}.    Similarly,  for a matrix  $W\in\mathbb{K}^{n\times n}$ such that $WW^{[H]}=\Phi$, where $\Phi$ is $H$-selfadjoint involutory with a hyperbolic negative eigenspace, there exists a left decomposition
\begin{align}
W = X L,
\label{eq:introxl}
\end{align}
where $L\in \mathbb{K}^{n\times n}$ is $H$-unitary and $X\in \mathbb{K}^{n\times n}$ is $H$-normal $H$-neutral involutory.  The decompositions~(\ref{eq:introlx}) and~(\ref{eq:introxl}) are not unique. 

We  show that $W$ in~(\ref{eq:F=WS=SW})  satisfies the conditions for the decompositions~(\ref{eq:introlx}) and~(\ref{eq:introxl}).   Therefore,  any nonsingular square matrix $F\in\mathbb{K}^{n\times n}$ can be factorized into a product of at most three $H$-normal matrices
\begin{align}
F = L X S,
\label{eq:introlxs}
\end{align}
where $L\in \mathbb{K}^{n\times n}$ is $H$-unitary, $X\in \mathbb{K}^{n\times n}$ is $H$-normal $H$-neutral involutory, and $S\in \mathbb{K}^{n\times n}$ is $H$-selfadjoint  and  $r$-positive-definite. Other similar decompositions exist,
\begin{align}
F = S' L_1 X_1= S' X' L' =X'_1 L'_1 S ,
\label{eq:introsxl}
\end{align}
where $L_1,L',L'_1\in \mathbb{K}^{n\times n}$ are $H$-unitary, $X_1,X',X'_1\in \mathbb{K}^{n\times n}$ are $H$-normal $H$-neutral involutory, and $S,S'\in \mathbb{K}^{n\times n}$ are $H$-selfadjoint and $r$-positive-definite. The factors $S$ and $S'$ are uniquely determined by $F$ (given $H$),  while the other matrices satisfy $LX=L_1X_1=X'L'=X'_1L'_1$.

In Section~\ref{sec:sign_function}, we define a sign matrix function of a matrix $A\in\mathbb{K}^{n\times n}$.  In Section~\ref{sec:canonical}, we review the canonical form of a pair $(A, H)$, where $H\in\mathbb{K}^{n\times n}$ is a selfadjoint matrix and  $A\in\mathbb{K}^{n\times n}$ is an $H$-selfadjoint matrix.   In Section~\ref{sec:neutral}, we introduce $H$-normal $H$-neutral involutory matrices and give some of their properties.  In Section~\ref{sec:W=LX}, we present the factorizations $W=LX$ and $W=XL$ in~(\ref{eq:introlx}) and~(\ref{eq:introxl}) respectively.   In Section~\ref{sec:F=LXS},  we present  the decompositions $F = L X S= S' L_1 X_1= S' X' L' =X'_1 L'_1 S,
$ in (\ref{eq:introlxs}) and (\ref{eq:introsxl}) for any nonsingular square matrix $F\in\mathbb{K}^{n\times n}$.

\section{A sign function}
\label{sec:sign_function}

We start by recalling some facts about primary matrix functions (see, e.g.,~\cite{Gantmacher1977, Higham2008,HornJohnson1991}). A primary matrix function $f$ of a matrix $A\in\mathbb{K}^{n\times n}$ can be defined by means of a function $f: \mathbb{C} \to \mathbb{C}$ (denoted by the same letter) defined on the spectrum of $A$ and called the stem function of the matrix function $f$.

\begin{defn}(chapter~V in~\cite{Gantmacher1977} or Definition 1.1 in~\cite{Higham2008}) 
A  function $f: \mathbb{C} \to \mathbb{C}$ is said to be defined on the spectrum of a matrix $A\in\mathbb{K}^{n\times n}$ if its value $f(\lambda_k)$  and the values of its $s_k-1$ derivatives 
\begin{align}
\label{eq:derivative}
f^{(j)}(\lambda_k),\quad \quad \quad j=0, \hdots,  s_k-1, \quad \quad k =1, \hdots, t,
\end{align}
exist at all eigenvalues  $\lambda_k$ of $A$.  Here $s_k$  is the size of the Jordan blocks $J_{s_k}(\lambda_k)$ in the Jordan decomposition of $A$.
\end{defn} 
As remarked in~\cite{Higham2008} right after Definition 1.1, arbitrary numbers can be chosen  and assigned as the values of $f(\lambda_k)$ and its derivatives $f^{(j)}(\lambda_k), j=1,\hdots, s_k-1$, at each eigenvalue $\lambda_k$ of $A$.

A primary matrix function $f(A)$  of a matrix $A\in\mathbb{K}^{n\times n}$ is well defined in the sense that it is unique. 
Since every primary matrix function of $A$ is a polynomial in $A$ (see Thm.~1.12 in~\cite{HornJohnson1991}), all primary matrix functions $f(A)$ commute with the matrix $A$  and also  commute with  each other. 

Moreover, $f(A^T) = f(A)^T$ and $f(Q^{-1} A Q) = Q^{-1} f(A) Q$ for a nonsingular matrix $Q\in\mathbb{K}^{n\times n}$ hold for any primary matrix function $f$ of a matrix $A\in\mathbb{K}^{n\times n}$.  It follows that for a real symmetric inner product defined by a real symmetric matrix $H\in \mathbb{R}^{n\times n}$,   $f(A^{[H]})=f(A)^{[H]}$ always holds, while for a complex Hermitian inner product defined by a complex Hermitian matrix $H\in \mathbb{C}^{n\times n}$,   $f(A^{[H]})=f(A)^{[H]}$ if and only if $f(\bar{A})=\overline{f(A)}$ (see Thm.~3.1 in~\cite{Higham2005}).    If the stem function in Equation~(\ref{eq:derivative}) satisfies $f^{(j)}(\bar{\lambda})=\overline{f^{(j)}(\lambda)}$, then $f(\bar{A})=\overline{f(A)}$, and $f(A)$ is real when $A$ is real  (see e.g.~\cite{HornJohnson1991}).

For an $H$-selfadjoint matrix $A\in\mathbb{K}^{n\times n}$, if $f(A^{[H]})=f(A)^{[H]}$, then both $f(A)$ and $f(A)A$ are $H$-selfadjoint.

The matrix sign function in Roberts~\cite{Roberts1971},  commonly used in the mathematical literature on control theory, eigendecompositions, and roots of matrices (see, e.g.,~\cite{Higham1994,KenneyLaub1995}),  is defined as the primary matrix function associated to the scalar  function
\begin{align}
f(\lambda)=\begin{cases}
+1, \quad\quad & \text{for} \quad \Re \lambda>0, \\
-1, \quad\quad &\text{for}  \quad  \Re \lambda<0 , \\
\text{undefined},  \quad\quad &\text{for}  \quad  \Re \lambda=0.
\end{cases}
\end{align}

Here we introduce a different sign function of a nonsingular  matrix $A$ through a scalar function  $\Sign(\lambda)$ as follows. 
\begin{defn}
The function $\Sign$ is defined as 
\begin{align}\label{eq:Sign}
\Sign(\lambda)=\begin{cases}
\text{undefined}, \quad \quad &\text{for}\quad  \lambda=0,\\
-1, \quad \quad &\text{for}\quad  \Re \lambda <0,  \Imag \lambda=0, \\
+1, \quad\quad &\text {otherwise},
\end{cases}
\end{align}
and  all derivatives of $\Sign(\lambda)$ at  $\lambda$ are equal to zero, i.e.,
\begin{align}
\Sign{^{(j)}}(\lambda)=0,  \quad \quad \quad j\ge 1.
\end{align}
\end{defn}

With the stem  function  (\ref{eq:Sign}), one can define the corresponding matrix sign function through the definition of primary function  (see, e.g.,~\cite{Gantmacher1977, Higham2008, HornJohnson1991}).   The matrix sign function $\Sign(A)$ is a primary matrix function of matrix $A \in \mathbb{K}^{n\times n}$ and thus is unique.

For a general complex matrix $A$, $\Sign(A)$ is complex in general.   For a real matrix $A$, $\Sign(A)$ is real, since $\overline{\Sign(\lambda)}=\Sign{\bar{\lambda}}$.

The matrix $\Sign(A)$ is an involutory  matrix, i.e., 
\begin{align}
\left[\text{Sign} (A) \right]^2=I_n.
\end{align}
The negative eigenspace of $\Sign(A)$ (i.e., the eigenspace with eigenvalue $-1$) is  the negative eigenspace of $A$.  The positive eigenspace of $\Sign(A)$ (i.e., the eigenspace with eigenvalue $+1$) is the sum of the positive eigenspace and nonreal eigenspace of $A$.

\section{Canonical form of a pair $(A, H)$} 
\label{sec:canonical}
In this section, we review some facts about the canonical form of a pair of matrices $(A,H)$, where $H \in \mathbb{K}^{n\times n}$ is a selfadjoint matrix and $A\in \mathbb{K}^{n\times n}$ is an $H$-selfadjoint matrix.   Then for a nonsingular matrix $F\in \mathbb{K}^{n\times n}$,  we present a proposition of an  $H$-selfadjoint matrix  $F^{[H]}F$ through the canonical form of $(F^{[H]}F, H)$.

\begin{defn}(Unitarily  similar pairs  in \cite{Gohberg2005} pp.~55 and~133)
Let  $H_1, H_2\in \mathbb{K}^{n \times n}$ be invertible selfadjoint matrices.  Let $A_1,A_2 \in \mathbb{K}^{n \times n}$  be two $n\times n$ matrices. 
The pairs $(A_1,H_1)$ and $(A_2, H_2)$ are said  to be unitarily similar (for $\mathbb{K=R}$, also called  r-unitarily  similar or orthogonally similar) if there exists an invertible matrix $Q\in \mathbb{K}^{n\times n}$  such that \begin{align}
\label{eq:transformation}
A_1=Q^{-1}A_2Q
\quad \text{and} \quad
H_1=Q^\dagger H_2Q.
\end{align}
\end{defn}
If $(A_1,H_1)$ and $(A_2,H_2)$ are unitarily similar, it follows that if $A_1$ is $H_1$-selfadjoint, then $A_2$ is $H_2$-selfadjoint,  and that  if $A_1$ is $H_1$-unitary, then $A_2$ is $H_2$-unitary.   

If $H_1=H_2$, then the transformation matrix $Q$ is  an $H$-unitary matrix.
\begin{defn}($H$-unitarily similar  matrices  in \cite{Gohberg2005} pp.~83 and~133)
\label{defn:u_similar}
Let $H\in \mathbb{K}^{n \times n}$ be an invertible  selfadjoint matrix. 
Two matrices $A_1,A_2 \in \mathbb{K}^{n \times n}$ are said  to be  $H$-unitarily similar (for $\mathbb{K}=\mathbb{R}$, also called  $H$-orthogonally similar) if there exists an $H$-unitary matrix $L$ such that
\begin{align}
A_1=L^{-1}A_2L.
\end{align}
\end{defn}

Let $Z_s$  be the $s \times s$ matrix
\begin{align}
Z_s=\begin{pmatrix}
&&&&1\\
&&&1&\\
&&\reflectbox{$\ddots$}&&\\
&1&&&\\
1&&&&\\
\end{pmatrix}.
\end{align}

\begin{thm}
\label{thm:canonical_C} (Complex canonical form; Thm.~5.1.1 in~\cite{Gohberg2005})  Let  $H \in \mathbb{C}^{n\times n}$ be a nonsingular  Hermitian matrix and  $A \in \mathbb{C}^{n\times n}$ be an $H$-Hermitian matrix.  Then the pair $(A, H)$ is unitarily similar to a canonical pair $(J, K)$ through an invertible  transformation $Q\in \mathbb{C}^{n\times n}$,  i.e.,
\begin{equation}
\label{eq:canonicalpair}
Q^{-1} A Q = J \quad \text{and} \qquad \bar{Q}^T H Q=K,
\end{equation}
where $J$ is the complex  Jordan form of $A$, namely
\begin{align}
J&=J_{s_1}(\lambda_1) \oplus \hdots \oplus J_{s_p}(\lambda_p) \oplus J_{s_{p+1}}(\lambda_{p+1}; \bar{\lambda}_{p+1}) \oplus \hdots \oplus J_{s_{p+q}}(\lambda_{p+q}; \bar{\lambda}_{p+q}),
\end{align}
with real eigenvalues $\lambda_1,\hdots, \lambda_p$ and nonreal eigenvalues $\lambda_{p+1}, \hdots, \lambda_{p+q}$,  and
\begin{align}
K&= \epsilon_1 Z_{s_1}\oplus  \hdots \oplus  \epsilon_p Z_{s_p} \oplus  Z_{s_{p+1}} \oplus \hdots \oplus Z_{s_{p+q}},
\end{align} 
with  $\epsilon_1=\pm 1, \hdots, \epsilon_p=\pm1$.  
\end{thm}
 
For a pair $(A, H)$ of real matrices $A, H\in \mathbb{R}^{n\times n}$,   there is a real canonical form  $(J, K)$ given by  the following theorem.
\begin{thm}
\label{thm:canonical_R} (Real canonical form; Thm.~6.1.5 in~\cite{Gohberg2005})
Let $H \in \mathbb{R}^{n\times n}$ be a nonsingular  symmetric matrix and $A\in \mathbb{R}^{n\times n}$ be an $H$-symmetric matrix.  Then the pair $(A, H)$ is orthogonally similar to a real canonical pair $(J, K)$ through an invertible real transformation matrix $Q\in \mathbb{R}^{n\times n}$,  i.e.,
\begin{equation}
\label{eq:canonicalpair_real}
Q^{-1} A Q = J\quad \text{and} \qquad Q^T H Q=K,
\end{equation}
where $J$ is the real Jordan form of $A$, namely
\begin{align}
J&=J_{s_1}(\lambda_1)\oplus \hdots \oplus J_{s_p}(\lambda_p) \oplus J_{s_{p+1}}(\alpha_1,\beta_1) \oplus \hdots \oplus J_{s_{p+q}}(\alpha_q, \beta_q),
\end{align}
with real eigenvalues $\lambda_1,\hdots, \lambda_p$ and nonreal eigenvalues $\alpha_1\pm i \beta_1, \hdots, \alpha_q\pm i \beta_q$,  and
\begin{align}
K&= \epsilon_1 Z_{s_1}\oplus  \hdots \oplus  \epsilon_p Z_{s_p} \oplus  Z_{s_{p+1}} \oplus \hdots \oplus Z_{s_{p+q}},
\end{align} 
with  $\epsilon_1=\pm 1, \hdots, \epsilon_p=\pm1$.
\end{thm}

Notice that $K$ has the same block structure as $J$. The canonical  form of a pair $(A,H)$ is unique up to the  order of the blocks.   The ordered set of signs $(\epsilon_1, \ldots, \epsilon_p)$ is called the sign characteristic of the pair $(A,H)$.  The signs $\epsilon_k \,(k=1,\hdots, p)$ are uniquely determined by $(A,H)$ up to permutation of signs in the blocks of $K$ corresponding to equal Jordan blocks of  $J$.

In Theorems~\ref{thm:canonical_C} and~\ref{thm:canonical_R}, $A\in \mathbb{K}^{n\times n}$ represents  any $H$-selfadjoint matrix.   

Let $H \in \mathbb{K}^{n\times n}$ be a nonsingular selfadjoint matrix and $F \in \mathbb{K}^{n\times n}$ be a  nonsingular matrix, then $F^{[H]}F$ is an $H$-selfadjoint matrix.   
By Theorems~2.1 and~2.2 in~\cite{Bolshakov1995},  $F^{[H]}F$ belongs to a particular  kind of $H$-selfadjoint matrices.

We now prove that the negative eigenspace of $F^{[H]}F$ and of $\Sign(F^{[H]}F)$ is a hyperbolic subspace. And similarly, the negative eigenspace  of $FF^{[H]}$ and of $\Sign(FF^{[H]})$ is a hyperbolic subspace. 

\begin{thm}
\label{thm:Sigma}
Let $H \in \mathbb{K}^{n\times n}$ be a nonsingular selfadjoint matrix and $F \in \mathbb{K}^{n\times n}$ be a  nonsingular matrix.  Then the negative eigenspace of  $F^{[H]}F$ and of  $\Sigma=\Sign(F^{[H]}F)$ is a  hyperbolic subspace.  
\end{thm}

\begin{proof}
The negative eigenspace of $F^{[H]}F$ is the subspace  spanned by all the generalized eigenvectors of $F^{[H]}F$ belonging to negative eigenvalues.  It coincides with the negative eigenspace of $\Sigma=\Sign(F^{[H]}F)$ through definition of the sign function. 

\begin{table}[t]
\caption{\label{tab:table1} Properties of possible block pairs $(J_s, K_s)$ that may appear  in the canonical form $(J,K)$ of $(F^{[H]}F, H)$,   where $s$ indicates the size of $J_s$ and $K_s$.
For a real eigenvalue $\lambda$,  $J_s=J_{s}(\lambda)$ and $K_s=\epsilon Z_{s}$, where $\epsilon$ is the sign characteristic corresponding to $\lambda$.  For a pair of complex conjugate eigenvalues $\lambda$ and  $\bar{\lambda}$,  the size $s$ is even, $K_s=Z_{s}$, and $J_s=J_{s}(\lambda; \bar{\lambda})$ for the complex form $(J,K)$,  $J_s=J_s(\alpha, \beta)$ with  $\lambda=\alpha \pm i\beta$ for the real form $(J,K)$.  Moreover,  $\sigma=\Sign(\lambda)$.  The last column indicates the number of blocks of kind $(J_s,K_s)$ appearing in $(J,K)$.} 
\begin{tabular}{ccccccc} 
\hline \hline 
$\lambda$ &$\sigma$& $\epsilon$ & $s$ & inertia of $K_s$ & signature of $K_s$  & $\begin{matrix}\text{number of}\\ \text{blocks}\end{matrix}$ \\[1ex]
\hline\\
$<0$& -1 & $+1$& odd &$ \left( \frac{s+1}{2}, \frac{s-1}{2} \right)$ &1 & $N^o_{-+}$ \\[1ex]
$<0$&-1&   $-1$& odd &$ \left( \frac{s-1}{2}, \frac{s+1}{2} \right)$ &-1 & $N^o_{--}$ \\[1ex]
$<0$&-1&  $+1$ &even & $ \left( \frac{s}{2}, \frac{s}{2} \right)$ & 0 & $N^e_{-+}$ \\[1ex]
$<0$&-1&   $-1$ & even & $ \left( \frac{s}{2}, \frac{s}{2} \right)$ &0 & $N^e_{--}$ \\[1ex]
$>0$  &+1&$+1$ & odd &$ \left( \frac{s+1}{2}, \frac{s-1}{2} \right)$ & 1 & $N^o_{++}$ \\[1ex]
 $>0$&+1 & $-1$ & odd & $ \left( \frac{s-1}{2}, \frac{s+1}{2} \right)$ &-1 & $N^o_{+-}$ \\[1ex]
$>0$&+1 &  $+1$ & even & $ \left( \frac{s}{2}, \frac{s}{2} \right)$ &0 & $N^e_{++}$ \\[1ex]
$>0$& +1& $-1$ &even & $ \left( \frac{s}{2}, \frac{s}{2} \right)$ &  0 & $N^e_{+-}$ \\[1ex]
nonreal  &+1 & - &  even & $ \left( \frac{s}{2}, \frac{s}{2} \right)$ &0& $N^n_{++}$ \\[1ex]
\hline \hline 
\end{tabular}
\end{table}
Let $(J,K)$ be  the  Jordan canonical form of $(F^{[H]}F, H)$ in Theorem~\ref{thm:canonical_C} or~\ref{thm:canonical_R}.
Table~\ref{tab:table1} lists the inertia and signature of each possible block  pair  in $(J,K)$.
By Theorems~2.1 and~2.2 in~\cite{Bolshakov1995},  for the  nonsingular matrix $F^{[H]}F$,  one has  
\begin{align}
\label{eq:N=N}
N^o_{-+}=N^o_{--},
\end{align}
where  $N^o_{-+}$ is the number of odd Jordan blocks with negative real eigenvalue and sign characteristic $+1$ and $N^o_{--}$ is the number of odd Jordan blocks with negative real eigenvalue and sign characteristic  $-1$.   
The  total dimension of the negative eigenspace of $F^{[H]}F$ is, from Table~\ref{tab:table1}, 
\begin{align}
n_-= \underbrace{s_h+\hdots+s_l}_{N^o_{-+}~\text{terms}}+\underbrace{s_i+\hdots +s_j}_{N^o_{--}}+\underbrace{s_u+\hdots+s_v}_{N^e_{-+}}+\underbrace{s_w+\hdots+s_z}_{N^e_{--}},
\end{align}
where the numbers under the braces indicate the number of terms in the respective sums.  Since (i) the sizes $s_u,\hdots,s_v$ and $s_w,\hdots,s_z$ are even integers, (ii) the sizes $s_h, \hdots, s_l$ and   $s_i, \hdots, s_j$ are odd integers, and (iii) $N^o_{-+}=N^o_{--}$, it follows that $n_-$ is an even integer. 
Similarly,  the signature of the  negative  eigenspace  of $F^{[H]}F$ is 
\begin{align}
\text{sig}_-=\underbrace{1+\hdots+1}_{N^o_{-+}~\text{terms}}+\underbrace{(-1)+\hdots +(-1)}_{N^o_{--}}+\underbrace{0+\hdots+0}_{N^e_{-+}}+\underbrace{0+\hdots+0}_{N^e_{--}}=0,
\end{align}
since  $N^o_{-+}=N^o_{--}$.  In addition,  each matrix $\epsilon_k Z_{s_k}$ is nonsingular, and thus its corresponding invariant subspace is nondegenerate.  Therefore the negative eigenspace of $F^{[H]}F$ is  a  hyperbolic subspace. 
\end{proof}

The proof that the negative eigenspace of $\Sigma^\prime=\Sign(FF^{[H]})$  is a hyperbolic subspace for a nonsingular matrix  $F\in \mathbb{K}^{n\times n}$ follows  from Theorem~\ref{thm:Sigma} by replacing $F$ with $F^{[H]}$.

\section{$H$-normal $H$-neutral involutory matrices}
\label{sec:neutral}

In this section we define $H$-normal $H$-neutral involutory matrices, and give some of their basic properties together with some canonical forms.

Let $H \in \mathbb{K}^{n\times n}$ be a nonsingular selfadjoint matrix.  We recall that two subspaces ${\cal U}, \ {\cal V} \subseteq \mathbb{K}^n$ are said to be orthogonal (or $H$-orthogonal) to each other if $[u,v]_H=0$ for all $u\in{\cal U}$ and $v\in{\cal V}$. If ${\cal U}$ and ${\cal V}$ are orthogonal subspaces we write ${\cal U} \perp {\cal V}$. 

We also recall that a subspace ${\cal N} \subset \mathbb{K}^n$ is called neutral (or $H$-neutral)  if $[u,v]_H=0$ for all $u,v\in {\cal N}$, i.e., \ ${\cal N} \perp {\cal N}$.  An equivalent definition  of a neutral subspace ${\cal N}$ is $[u,u]_H=0$ for all $u\in {\cal N}$. 

\subsection{Involutions}
A matrix $X\in \mathbb{K}^{n\times n}$ is involutory if $X^2=I_n$. Any involutory matrix is diagonalizable,  and its eigenvalues are $+1$ and $-1$.
For an involutory matrix $X\in \mathbb{K}^{n\times n}$, let 
\begin{align}
\Pos(X)= \{ v \in \mathbb{K}^n | Xv= v\}
\end{align}
indicate the positive eigenspace of  $X$, and let 
\begin{align}
\Neg(X)= \{ v \in \mathbb{K}^n | Xv= -v\}
\end{align}
indicate the negative eigenspace of  $X$.

\begin{lem}
\label{thm:XH_involutory}
Let $H \in \mathbb{K}^{n\times n}$ be a nonsingular selfadjoint matrix and let $X\in \mathbb{K}^{n\times n}$ be an involutory matrix. Then $X^{[H]}$ is also an involutory matrix. 
\end{lem}
\begin{proof}
By direct calculation, one has $(X^{[H]})^2=(X^2)^{[H]}=I_n^{[H]}=I_n$. Therefore  $X^{[H]}$ is an involutory matrix. 
\end{proof}

The projection matrix onto $\Neg (X)$ is $P_X=(I_n-X)/2$.  The projection matrix onto $\Pos (X)$ is $P_X^\prime=(I_n+X)/2$. One has 
\begin{align}\label{eq:P_X}
\Im (P_X)=\Ker (P_X^\prime) =\Neg (X), \quad \quad  \Ker (P_X)=\Im (P_X^\prime)=\Pos(X).
\end{align}

\begin{lem}
\label{thm:pos_neg}
Let $H \in \mathbb{K}^{n\times n}$ be a nonsingular selfadjoint matrix and let $X \in \mathbb{K}^{n \times n}$ be an involutory matrix. Then
\\\indent(a)  $\Pos (X)=(\Neg X^{[H]})^\perp$,
\\\indent(b)  $\Neg (X)=(\Pos X^{[H]})^\perp$,
\\\indent(c)  $\Pos (X^{[H]})=(\Neg X)^\perp$,
\\\indent(d) $\Neg (X^{[H]})=(\Pos X)^\perp$.
\end{lem}

\begin{proof}Let $P_X=(I_n-X)/2$. By~(\ref{eq:P_X}),  one has $\Ker (P_X)=\Pos (X)$ and $\Im (P_X)=\Neg (X)$.
Since $X$ is involutory, by Lemma~\ref{thm:XH_involutory}, $X^{[H]}$ is involutory and $P_X^{[H]}=(I_n-X^{[H]})/2$, one has  $\Ker (P_X^{[H]})=\Pos (X^{[H]})$ and $\Im (P_X^{[H]})=\Neg (X^{[H]})$.

(a, b) Since $\Ker (P_X)=(\Im P_X^{[H]})^\perp$ and $\Im(P_X)= (\Ker P_X^{[H]})^\perp$ (see, e.g., Proposition 4.1.1 in~\cite{Gohberg2005}),  we have  $\Pos (X)=(\Neg X^{[H]})^\perp$ and $\Neg (X)=(\Pos X^{[H]})^\perp$.  

(c, d)  Replace $X$ with $X^{[H]}$ in (a, b).
\end{proof}

\subsection{$H$-neutral involutions} 
\begin{defn}
\label{defn:neutral_involution}
($H$-neutral involutory matrix) Let $H \in \mathbb{K}^{n\times n}$ be a nonsingular selfadjoint matrix.   An involutory matrix $X\in \mathbb{K}^{n\times n}$ is called $H$-neutral if its negative eigenspace is  $H$-neutral or $\{0 \}$.
The neutral index $m_X$ of  $X$  is defined as the dimension of the negative eigenspace $\Neg(X)$, i.e.,  $m_X=\dim(\Neg X)$.     The identity matrix $I_n$ is $H$-neutral involutory of neutral index $0$.
\end{defn}

\begin{lem}
\label{thm:NegX_PosXH}
Let $H \in \mathbb{K}^{n\times n}$ be a nonsingular selfadjoint matrix and let $X\in \mathbb{K}^{n\times n}$ be an $H$-neutral involutory matrix.   Then  $\Neg(X)\subseteq \Pos(X^{[H]})$.
\end{lem}

\begin{proof}
Since $\Neg (X)$ is neutral, $\Neg(X)\subseteq (\Neg X)^\perp$.  By Lemma~\ref{thm:pos_neg}(c), $(\Neg X)^\perp=\Pos(X^{[H]})$. Therefore $\Neg(X)\subseteq \Pos(X^{[H]})$.
\end{proof}

\begin{lem}
\label{thm:NegX_NegXH}
Let $H \in \mathbb{K}^{n\times n}$ be a nonsingular selfadjoint matrix and let $X\in \mathbb{K}^{n\times n}$ be  an $H$-neutral involutory matrix. 
Then $\Neg(X)\cap \Neg(X^{[H]})=\{0\}$.
\end{lem}

\begin{proof}
By Lemma~\ref{thm:NegX_PosXH}, $\Neg(X)\subseteq \Pos(X^{[H]})$. Since $\Pos(X^{[H]})\cap \Neg(X^{[H]})=\{0\}$,   it  follows that  $\Neg(X)\cap \Neg(X^{[H]})=\{0\}$.
\end{proof}

\begin{pro}
\label{thm:H_neutral}
Let $H \in \mathbb{K}^{n\times n}$ be a nonsingular selfadjoint matrix.  A matrix $X \in \mathbb{K}^{n\times n}$ is $H$-neutral involutory if and only if  $X^2=I_n$ and $X^{[H]}X=X^{[H]}+X-I_n$.
\end{pro}

\begin{proof}  Sufficiency: Let an involutory matrix $X\in \mathbb{K}^{n\times n}$,  satisfy $X^{[H]}X=X^{[H]}+X-I_n$.  Let  $u,v \in  \Neg X$.  Then  
\begin{align}
[u,v]_H=&[-Xu,-Xv]_H=[Xu,Xv]_H=[u,X^{[H]}X v]_H \nonumber \\
=&[u, (X^{[H]}+X-I_n)v]_H=[Xu,v]_H+[u,Xv]_H-[u,v]_H \nonumber \\
=&[-u,v]_H+[u,-v]_H-[u,v]_H=-3[u,v]_H,
\end{align}
thus $[u,v]_H=0$, i.e., $\Neg (X)$ is $H$-neutral.   By Definition~\ref{defn:neutral_involution},  $X$ is $H$-neutral involutory.

Necessity:   Let $X\in \mathbb{K}^{n\times n}$ be an $H$-neutral involutory matrix, then  $X^2=I_n$. By Lemma~\ref{thm:NegX_PosXH}, $\Neg(X)\subseteq \Pos(X^{[H]})$.
It follows that, using~(\ref{eq:P_X}), $\Im (I_n-X) = \Im (P_X)= \Neg X\subseteq \Pos (X^{[H]})=\Ker (P_X^{[H]})=\Ker (I_n-X^{[H]})$. Thus   $(I_n-X^{[H]})(I_n-X)v=0$ for all  $v \in \mathbb{K}^{n}$. Therefore $X^{[H]}X-X^{[H]}-X+I_n=0$.
\end{proof}

\subsection{$H$-normal $H$-neutral involutions}

\begin{defn}
($H$-normal $H$-neutral involutory matrix) Let $H \in \mathbb{K}^{n\times n}$ be a nonsingular selfadjoint matrix.   An involutory matrix $X\in \mathbb{K}^{n\times n}$ is called $H$-normal $H$-neutral if it is both $H$-neutral and  $H$-normal.
\end{defn}

\begin{thm} 
\label{thm:X_XH}
Let $H \in \mathbb{K}^{n\times n}$ be a nonsingular selfadjoint matrix and let $X \in \mathbb{K}^{n\times n}$ be an $n\times n$ matrix.  The following statements  are equivalent.
\\\indent(a) $X$ is $H$-normal $H$-neutral involutory.  
\\\indent(b) $X$ and $X^{[H]}$ are $H$-neutral involutory. 
\\\indent(c) $X^2=I_n$ and $X^{[H]}X=XX^{[H]}=X^{[H]}+X-I_n$.
\end{thm}

\begin{proof}
(a)$\to$(b):  Since $X$ is involutory, by Lemma~\ref{thm:XH_involutory},  $X^{[H]}$ is involutory.  Since $X$ is $H$-neutral,  by Proposition~\ref{thm:H_neutral}, $X^{[H]}X=X^{[H]}+X-I_n$. Since $X$ is $H$-normal,  $X^{[H]}X=XX^{[H]}$. Therefore, $XX^{[H]}=X+X^{[H]}-I_n$. By Proposition~\ref{thm:H_neutral} with $X$ replaced by $X^{[H]}$, $X^{[H]}$ is $H$-neutral.

(b)$\to$(c): Since $X$ is $H$-neutral  involutory,  by Proposition~\ref{thm:H_neutral},  $X^2=I_n$ and $X^{[H]}X=X^{[H]}+X-I_n$. Since $X^{[H]}$ is $H$-neutral  involutory, by Proposition~\ref{thm:H_neutral},  $XX^{[H]}=X^{[H]}+X-I_n$.
Combining the equalities, we have $X^{[H]}X=XX^{[H]}=X^{[H]}+X-I_n$.

(c)$\to$(a): Since $X^{[H]}X=XX^{[H]}$, $X$ is $H$-normal.  Since $X^2=I_n$ and  $X^{[H]}X=X^{[H]}+X-I_n$, by Proposition~\ref{thm:H_neutral}, $X$ is $H$-neutral  involutory.  So $X$ is $H$-normal $H$-neutral involutory.  
\end{proof}

\begin{pro}
\label{thm:xh_is_neutral_involutory}
Let $H \in \mathbb{K}^{n\times n}$ be a nonsingular selfadjoint matrix. If $X\in \mathbb{K}^{n\times n}$ is $H$-normal $H$-neutral involutory,  so is its $H$-adjoint $X^{[H]}$. Moreover, $X$ and $X^{[H]}$ have the same neutral index.
\end{pro}

\begin{proof}
That $X^{[H]}$ is $H$-normal $H$-neutral involutory follows directly from Theorem~\ref{thm:X_XH}. That $X$ and $X^{[H]}$ have the same neutral index follows from the fact that $X^{[H]}=H^{-1} X^\dagger H$ is similar to $X^\dagger$, and $X^\dagger$ has the same eigenvalues, with the same multiplicities, as $X$, since the eigenvalues of $X$ are real.
\end{proof}

\begin{pro}
\label{thm:Sigma_2m}
Let $H \in \mathbb{K}^{n\times n}$ be a nonsingular selfadjoint matrix.
If $X \in \mathbb{K}^{n \times n}$ is an $H$-normal $H$-neutral involutory matrix of neutral index $m$, then the $H$-selfadjoint matrix  $X^{[H]}X$ is involutory and its negative eigenspace is hyperbolic of dimension $2m$.
\end{pro}

\begin{proof} 
By direct calculation, $(X^{[H]}X)^2=X^{[H]}X X^{[H]}X=X^{[H]}X^{[H]}XX=I_n$, so $X^{[H]}X$ is involutory. It follows that $X^{[H]}X$  has eigenvalues $+1$ and $-1$ only. Let $n_+$ and $n_-$  denote the number of positive (i.e., $+1$) and negative (i.e., $-1$) eigenvalues. One has $n_++n_-=n$ and $\tr(X^{[H]}X)=n_+-n_-$. On the other hand, since $X$ is involutory of neutral index $m$, it has $m$ negative eigenvalues equal to $-1$ and $n-m$ positive eigenvalues equal to $+1$. Thus $\tr(X) = (n-m)-m=n-2m$. Hence, using $\tr(X^{[H]})=\tr(X)$ and $X^{[H]}X=X^{[H]}+X-I_n$, one has
\begin{align}
\tr(X^{[H]}X)=&\tr(X^{[H]}+X-I_n)=\tr(X^{[H]})+\tr(X)-n \nonumber \\
=&n-2m+n-2m-n=n-4m.
\end{align}
Thus, from $n_++n_-=n$ and $\tr(X^{[H]}X)=n_+-n_-=n-4m$, it follows that $n_-=2m$, that is the negative eigenspace of $X^{[H]}X$ has dimension $2m$.  Moreover, by Theorem~\ref{thm:Sigma},  the negative eigenspace  of $X^{[H]}X$ is hyperbolic. Therefore, the negative eigenspace  of $X^{[H]}X$ is hyperbolic of dimension $2m$.
\end{proof}

We now prove two lemmas to be used in the derivation of the canonical forms of an $H$-normal $H$-neutral involutory matrix $X$.

\begin{lem}
\label{thm:XXS_subspace} 
Let $H \in \mathbb{K}^{n\times n}$ be a nonsingular selfadjoint matrix, let $X \in \mathbb{K}^{n\times n}$ be an $H$-normal $H$-neutral involutory matrix, and let $\Phi=X^{[H]}X=XX^{[H]}$.  Then,
\\\indent(a)  $\Pos (\Phi)=\Pos (X) \cap \Pos (X^{[H]})$,
\\\indent(b) $\Neg(\Phi)=\Neg(X)\dot{+}\Neg(X^{[H]})$.
\end{lem}

\begin{proof} 
(a) Let $v \in \Pos(\Phi)$, then $\Phi v=v$, then $XX^{[H]} v=v$.  Multiplying both sides by $X$ on the left, and using 
$X^2=I_n$,  one has $X^{[H]}v=Xv$.  We now show that $Xv=v$.   Using Theorem~\ref{thm:X_XH}(c), calculate 
$v=XX^{[H]} v=X^{[H]}v+Xv-v=2Xv-v$,  from which $2Xv=2v$, $Xv=v$,  i.e.,  $v\in \Pos(X)$.   From $X^{[H]}v=Xv$, one has   $v\in \Pos(X^{[H]})$. Therefore,  $\Pos(\Phi) \subseteq \Pos(X) \cap \Pos(X^{[H]}) $.   On the other hand,  let $v\in \Pos(X) \cap \Pos(X^{[H]})$, then $Xv=v$ and $X^{[H]}v=v$.  Using Theorem~\ref{thm:X_XH}(c), calculate
$\Phi v=X^{[H]}Xv=(X^{[H]}+X-I_n)v=X^{[H]}v+Xv-v=v.$
That is $v \in \Pos (\Phi)$. Therefore,  $\Pos(X) \cap \Pos(X^{[H]}) \subseteq \Pos (\Phi)$. Combining the two results, one has  $\Pos (\Phi)=\Pos (X) \cap \Pos (X^{[H]})$.

(b) Assume $X$ has neutral index $m$, i.e., $\dim (\Neg X)=m$.  Let $v\in \Neg(X)$,  by Lemma~\ref{thm:NegX_PosXH},  $v\in \Pos(X^{[H]})$. Then  $\Phi v=X^{[H]}Xv=-X^{[H]}v=-v$,
i.e., $v\in \Neg(\Phi)$. Therefore $\Neg(X)\subseteq \Neg(\Phi)$.  Replacing  $X$ with $X^{[H]}$, one finds $\Neg(X^{[H]})\subseteq \Neg(\Phi)$.  Combining the two results, one has $ \Neg(X) + \Neg(X^{[H]})\subseteq \Neg(\Phi)$.  Since, by Lemma~\ref{thm:NegX_NegXH},  $\Neg(X) \cap \Neg(X^{[H]})=\{0\}$, then $\Neg(X) + \Neg(X^{[H]})=\Neg(X)\dot{+}\Neg(X^{[H]})$.  Since,  by Proposition~\ref{thm:xh_is_neutral_involutory},  $\dim (\Neg X^{[H]})=\dim (\Neg X)$,  then $\dim (\Neg X\dot{+}\Neg X^{[H]})=\dim (\Neg X) +\dim (\Neg X^{[H]})=2m$.  Since by Proposition~\ref{thm:Sigma_2m},  $\dim (\Neg \Phi)=2m$,  it follows $\Neg(\Phi)=\Neg(X)\dot{+}\Neg(X^{[H]})$.
\end{proof}

\begin{lem}
\label{thm:perp_subspace}
Let $H \in \mathbb{K}^{n\times n}$ be a nonsingular selfadjoint matrix, let $X \in \mathbb{K}^{n\times n}$ be an $H$-normal $H$-neutral involutory matrix, and let $\Phi=X^{[H]}X=XX^{[H]}$.  Then,
\\\indent(a) $\Neg(X) \perp \Pos(\Phi) $,
\\\indent(b)  $\Neg(X^{[H]})\perp \Pos(\Phi)$.
\end{lem}

\begin{proof}
$\Phi$ is $H$-selfadjoint, so $\Neg(\Phi) \perp \Pos(\Phi)$.   By Lemma~\ref{thm:XXS_subspace}(b), $\Neg(X) \subseteq \Neg(\Phi)$ and $\Neg(X^{[H]}) \subseteq \Neg(\Phi)$ . Thus $\Neg(X) \perp \Pos(\Phi)$ and $\Neg(X^{[H]})\perp \Pos(\Phi)$ .
\end{proof}

\begin{thm}
\label{thm:XH_JK} 
Let $H \in \mathbb{K}^{n\times n}$ be a nonsingular selfadjoint matrix and $X \in \mathbb{K}^{n\times n}$ be an $H$-normal $H$-neutral involutory matrix. Then the pair  $(X,H)$ is unitarily similar to a canonical pair $(J,  K)$ through an invertible  transformation $Q\in \mathbb{K}^{n\times n}$,  i.e.,
\begin{align}
\label{eq:XH_JK}
Q^{-1}XQ=J \quad  \text{and} \quad Q^{\dagger}H Q=K, 
\end{align}
where $J \in \mathbb{K}^{n\times n}$ is the Jordan form of $X$, 
\begin{align}
\label{eq:JK1}
J=\begin{pmatrix}
-I_m &&&\\ &I_m &&\\&&I_{p-m}&\\&&&I_{q-m}
\end{pmatrix},
\end{align}
and $K\in \mathbb{K}^{n\times n}$ is the nonsingular selfadjoint matrix
\begin{align}
\label{eq:JK2}
K=\begin{pmatrix}
&Z_m&&\\Z_m &&&\\&&I_{p-m}&\\&&&-I_{q-m}
\end{pmatrix}.
\end{align}
Here $m$ is the neutral index of $X$ and $(p,q)$ is the inertia of $H$.
\end{thm}

\begin{proof}Since $X$ is $H$-neutral,  the negative eigenspace $\Neg(X)$ is neutral.  By the definition of neutral index, the dimension of $\Neg(X)$ is $m$. Let  $v_1, \hdots, v_m$ be a basis of $\Neg (X)$, then $[v_i,v_j]_H=0 \, (i,j=1,\hdots, m)$. 
By Proposition~\ref{thm:xh_is_neutral_involutory}, $X^{[H]}$ is $H$-neutral.  The negative eigenspace $\Neg(X^{[H]})$ is neutral and the dimension of  $\Neg(X^{[H]})$ is $m$. Let $u_1,\hdots, u_m$ be a basis of $\Neg(X^{[H]})$,  then $[u_i,u_j]_H=0 \,(i,j=1,\hdots, m)$.  
By Proposition~\ref{thm:Sigma_2m},  the dimension of the negative eigenspace $\Neg(\Phi)$, where $\Phi=X^{[H]}X$, is $2m$. Thus the dimension of  the positive eigenspace $\Pos(\Phi)$ is $n-2m$.  Let $w_1,\hdots, w_{n-2m}$ be a basis of $\Pos(\Phi)$. By Lemma~\ref{thm:XXS_subspace}(a),  $w_1,\hdots, w_{n-2m} \in \Pos(X)$.   By Lemma~\ref{thm:perp_subspace}, $\Neg(X)\perp \Pos(\Phi)$ and $\Neg(X^{[H]})\perp \Pos(\Phi)$, i.e., $[v_i, w_j]_H=0 \,(i=1,\hdots, m; j=1,\hdots,n-2m)$ and  $[u_i, w_j]_H=0\, (i=1,\hdots, m; j=1,\hdots,n-2m)$.

Let $E_1$ be the $n\times n$ matrix with columns equal to the components of the basis vectors $v_1, \hdots v_m, u_1,\hdots, u_m, w_1,\hdots, w_{n-2m}$, in this order. In this basis, the matrices $X$ and $H$ assume the form
\begin{align}
E_1^{-1} X E_1 =
\begin{pmatrix}
-I_m &&\\ &I_m &\\&&I_{n-2m}
\end{pmatrix}
\quad \text{and} \quad
E_1^\dagger H E_1 =
\begin{pmatrix}
&B&\\B^\dagger &&\\&&C
\end{pmatrix}
,
\end{align}
where $B \in \mathbb{K}^{m\times m}$ and $C\in\mathbb{K}^{(n-2m)\times(n-2m)}$ are the matrices with elements
\begin{align}
B_{ij} = v_i^\dagger H u_j \quad (i=1,\hdots,m; j=1,\hdots,m),
\end{align}
and
\begin{align}
C_{ij} = w_i^\dagger H w_j \quad (i=1,\hdots,{n-2m}; j=1,\hdots,{n-2m}).
\end{align}

Since $H$ is nonsingular, also $B$  and $C$ are nonsingular. Moreover, $C$ is selfadjoint, and since $\Neg(\Phi) = \mathop{\rm span}(v_1, \hdots v_m, u_1,\hdots, u_m)$ is a hyperbolic subspace, the signature of $C$ is equal to the signature of $H$, i.e.,  $p-q$. Thus $C$ can be diagonalized as (recall that $C$ has dimension $n-2m$ and that $p+q=n$)
\begin{align}
C= D^\dagger \, \eta_{p-m,q-m} \, D ,
\end{align}
where $\eta_{p-m,q-m} \in \mathbb{K}^{(n-2m)\times(n-2m)}$ is the diagonal matrix
\begin{align}
\eta_{p-m,q-m} =  \begin{pmatrix} I_{p-m} & \\ & -I_{q-m} \end{pmatrix},
\end{align}
and $D \in \mathbb{K}^{(n-2m)\times(n-2m)}$ is a nonsingular $(n-2m)\times(n-2m)$ matrix.
Let
\begin{align}
E_2 =
\begin{pmatrix}
I_m&&\\ &B^{-1}&\\&&D^{-1}
\end{pmatrix},
\end{align}
and let $E'=E_1 E_2$.
Then, by direct computation,
\begin{align}
E^{\prime-1} X E' & =
\begin{pmatrix}
-I_m &&\\ &I_m &\\&&I_{n-2m}
\end{pmatrix},
\end{align}
and
\begin{align}
E^{\prime\dagger} H E' & =
\begin{pmatrix}
&I_m&\\I_m &&\\&&\eta_{p-m,q-m}
\end{pmatrix}
.
\end{align}
Let
\begin{align}
E_3 =
\begin{pmatrix}
I_m&&\\ &Z_m&\\&&I_{n-2m}
\end{pmatrix},
\end{align}
and let $Q=E' E_3=E_1 E_2 E_3$.
Then,
\begin{align}
Q^{-1} X Q & =
\begin{pmatrix}
-I_m &&\\ &I_m &\\&&I_{n-2m}
\end{pmatrix}
= J,
\end{align}
and
\begin{align}
Q^{\dagger} H Q & =
\begin{pmatrix}
&Z_m&\\Z_m &&\\&&\eta_{p-m,q-m}
\end{pmatrix}
= K,
\end{align}
where $J$ and $K$ are the matrices in the statement of the theorem.
\end{proof}

\begin{cor}
\label{thm:XH_PM} 
Let  $H \in \mathbb{K}^{n\times n}$ be a nonsingular selfadjoint matrix and $X \in \mathbb{K}^{n\times n}$ be an $H$-normal $H$-neutral involutory matrix.  Then the  pair $(X,H)$ is unitarily similar to a canonical pair $(P,M)$ through an invertible  transformation $Q'\in \mathbb{K}^{n\times n}$,  i.e.,
\begin{align}
\label{eq:XH_PM}
Q^{\prime -1} X Q^\prime =P  \quad  \text{and} \quad Q^{\prime \dagger} H Q^{\prime}=M, 
\end{align}
where $P \in \mathbb{K}^{n\times n}$ is the $M$-normal $M$-neutral involutory matrix 
\begin{align}
\label{eq:PM1}
P=\begin{pmatrix}
&Z_m&&\\Z_m &&&\\&&I_{p-m}&\\&&&I_{q-m}
\end{pmatrix},
\end{align}
and $M\in \mathbb{K}^{n\times n}$ is the nonsingular selfadjoint matrix
\begin{align}
\label{eq:PM2}
M=\begin{pmatrix}
-I_m &&&\\ &I_m &&\\&&I_{p-m}&\\&&&-I_{q-m}
\end{pmatrix}.
\end{align}
Here $m$ is the neutral index of $X$ and $(p,q)$ is the inertia of $H$.
\end{cor}

\begin{proof}
Let 
\begin{align}
E=\begin{pmatrix}
-I_m/\sqrt{2} &\,\,\,\,Z_m/\sqrt{2} &&\\
Z_m/\sqrt{2}&\,\,\,\,I_m/\sqrt{2} &&\\
&&I_{p-m}&\\
&&&I_{q-m}
\end{pmatrix}.
\end{align} 
Then, for $J$ and $K$ as in Theorem~\ref{thm:XH_JK},
\begin{align}
E^{-1}JE=P  \quad \text{and} \quad E^\dagger K E=M.
\end{align}
The choice $Q^\prime=QE$, with $Q$  as in Theorem~\ref{thm:XH_JK}, proves the corollary.
\end{proof}

The forms of $X^{[H]}$ and $\Phi=X^{[H]} X$ in the same basis as in the canonical forms of $(X,H)$ in Theorem~\ref{thm:XH_JK} and Corollary~\ref{thm:XH_PM} respectively  are
\begin{align}
Q^{-1}X^{[H]}Q = J^{[K]} = & \begin{pmatrix}
I_m &&&\\ &-I_m &&\\&&I_{p-m}&\\&&&I_{q-m}
\end{pmatrix},
\\
Q^{\prime -1}X^{[H]} Q' = P^{[M]} = & \begin{pmatrix}
&-Z_m&&\\-Z_m &&&\\&&I_{p-m}&\\&&&I_{q-m}
\end{pmatrix},
\\
Q^{-1}X^{[H]} X Q = Q^{\prime -1}X^{[H]} X Q' = &\begin{pmatrix}
-I_m&&&\\&-I_m &&\\&&I_{p-m}&\\&&&I_{q-m}
\end{pmatrix}.
\end{align}

The first $2m$ rows and columns of the canonical forms $(J,K)$ and $(P,M)$ in Theorem~\ref{thm:XH_JK} and Corollary~\ref{thm:XH_PM} correspond to the negative subspace of $X^{[H]} X$. By decomposing it into $m$ hyperbolic planes, it is easy to see that alternative canonical forms of the pair $(X,H)$ are the pairs $(J',K')$ and $(P',M')$ where
\begin{align}
J'=&\underbrace{\begin{pmatrix}-1&\\&1
\end{pmatrix} \oplus \hdots \oplus \begin{pmatrix}-1&\\&1
\end{pmatrix}}_{m~\text{terms}} \oplus \begin{pmatrix} I_{p-m}&\\&I_{q-m}
\end{pmatrix},   \label{eq:J=QXQ} \\
K'=&\underbrace{\begin{pmatrix}&1\\1&
\end{pmatrix} \oplus \hdots \oplus \begin{pmatrix}&1\\1&
\end{pmatrix}}_{m~\text{terms}} \oplus  \begin{pmatrix} I_{p-m}&\\&-I_{q-m}
\end{pmatrix},   \label{eq:K=QHQ}   
\end{align}
and
\begin{align}
P'=&\underbrace{\begin{pmatrix}&1\\1&
\end{pmatrix} \oplus \hdots \oplus \begin{pmatrix}&1\\1&
\end{pmatrix}}_{m~\text{terms}} \oplus  \begin{pmatrix} I_{p-m}&\\&I_{q-m}
\end{pmatrix},  \label{eq:P=QXQ} \\
M'=&\underbrace{\begin{pmatrix}-1&\\&1
\end{pmatrix} \oplus \hdots \oplus \begin{pmatrix}-1&\\&1
\end{pmatrix}}_{m~\text{terms}} \oplus  \begin{pmatrix} I_{p-m}&\\&-I_{q-m}
\end{pmatrix}.   \label{eq:M=QHQ} 
\end{align}

\begin{pro}
\label{thm:lemma2}
Let $H \in \mathbb{K}^{n\times n}$ be a nonsingular selfadjoint matrix.  Let $X_1$ and $X_2$ be two $H$-normal $H$-neutral involutory  matrices.  Then  $X_1$ and $X_2$ have  the same neutral index if and only if they are  $H$-unitarily similar. 
\end{pro}

\begin{proof}
Sufficiency:  Since $X_1$ and $X_2$ are $H$-unitarily similar, they have the same number of negative eigenvalues, thus $X_1$ and $X_2$ have  the same neutral index.

Necessity:  Let $X_1$ and $X_2$ be two $H$-normal $H$-neutral involutory matrices having neutral index $m$.   The pairs $(X_1, H)$ and $(X_2, H)$ can be put into the same  canonical form $(J,K)$ in Theorem~\ref{thm:XH_JK} through some invertible transformations $Q_1$ and $Q_2$ respectively.  Then,
\begin{align} 
Q_1^{-1} X_1Q_1=Q_2^{-1} X_2 Q_2 \quad \text{and} \quad Q_1^\dagger HQ_1=Q_2^\dagger HQ_2.
\end{align}
So $X_1$ and $X_2$ are $H$-unitarily similar through the matrix  $Q_1Q_2^{-1}$. 
\end{proof}

\begin{cor}   
Let $H \in \mathbb{K}^{n\times n}$ be a nonsingular selfadjoint matrix.  If $X \in \mathbb{K}^{n\times n}$ is $H$-normal $H$-neutral involutory,  then $X$ and $X^{[H]}$ are $H$-unitarily similar. 
\end{cor}
\begin{proof}
By Proposition~\ref{thm:xh_is_neutral_involutory}, $X^{[H]}$ is $H$-normal $H$-neutral involutory of the same neutral index as $X$.  By Proposition~\ref{thm:lemma2},  $X$ and $X^{[H]}$ are $H$-unitarily similar.
\end{proof}

\begin{pro}
\label{thm:lemma1}
Let $H \in \mathbb{K}^{n\times n}$ be a nonsingular selfadjoint  matrix and  $X_1, X_2 \in \mathbb{K}^{n \times n}$ be  $H$-unitarily similar.  If $X_1$ is  $H$-normal $H$-neutral involutory, then $X_2$ is also $H$-normal $H$-neutral involutory and has the same neutral index.   
\end{pro}

\begin{proof} Since $X_1$ is  $H$-normal $H$-neutral involutory,  by Theorem~\ref{thm:X_XH},   $X_1^2=I_n$ and $X_1^{[H]}X_1=X_1X_1^{[H]}=X_1^{[H]}+X_1-I_n$.   If $X_1$ and $X_2$  are $H$-unitarily similar,   then there exists an $H$-unitary matrix $L$ such that  $X_1=L^{-1}X_2 L$.  Then  $X_2^2=(LX_1L^{-1})(LX_1L^{-1})=I_n$.  Moreover, 
\begin{align}
X_2^{[H]}X_2=&(LX_1^{[H]}L^{-1})(LX_1L^{-1})=L(X_1^{[H]} X_1) L^{-1},  \label{eq:x2x2} \\
X_2X_2^{[H]}=&(LX_1L^{-1})(LX_1^{[H]}L^{-1})=L(X_1X_1^{[H]} )L^{-1}, \\
X_2^{[H]}+X_2-I_n=&L(X_1^{[H]}+X_1-I_n)L^{-1}. \label{eq:x2+x2-1}
\end{align}
Through (\ref{eq:x2x2}-\ref{eq:x2+x2-1}), one has $X_2^{[H]}X_2=X_2X_2^{[H]}=X_2^{[H]}+X_2-I_n$,  thus $X_2$ is $H$-normal $H$-neutral involutory. By Proposition~\ref{thm:lemma2},  $X_2$ has the same neutral index as $X_1$.
\end{proof}

\section{$W=LX$ and $W=XL$ factorizations}
\label{sec:W=LX}

Let $H_1,H_2\in \mathbb{K}^{n\times n}$ be two congruent nonsingular selfadjoint matrices.  A matrix $W\in \mathbb{K}^{n\times n}$ is called $(H_1, H_2)$-unitary if \begin{align}
\label{eq:WHWH}
W^\dagger H_1W=H_2
\end{align}
(see~\cite{Bolshakov1995, Gohberg2005}, where they are referred to as $H_2$-$H_1$-unitary or $(H_2, H_1)$-unitary).  This terminology arises  from
\begin{align}
[Wx,Wy]_{H_1}=[x,y]_{H_2}\quad  \text{for all} \quad x,y \in \mathbb{K}^n.
\end{align}
Note that any $(H_1, H_2)$-unitary matrix so defined is nonsingular.  Equation~(\ref{eq:WHWH}) can be written  as $W^{[H_1]}W=H_1^{-1}H_2$.   By Theorem~\ref{thm:Sigma},  the negative eigenspace of $H_1^{-1}H_2$ is hyperbolic.
  In particular, for an involutory  matrix $\Phi$ such that $H$ and $H\Phi$ are congruent,  $W$ is $(H, H\Phi)$-unitary if and only if $W^{[H]}W=\Phi$, and $W$ is $(H\Phi, H)$-unitary if and only if $WW^{[H]}=\Phi$.   

This  section is  devoted to  (non-unique) factorizations of $(H, H\Phi)$-unitary  and $(H\Phi, H)$-unitary  matrices, where  $\Phi$ is  an $H$-selfadjoint  involutory matrix  with a hyperbolic negative eigenspace.  An $(H, H\Phi)$-unitary  matrix $W$ can be factorized  into a product $LX$ of an $H$-unitary  matrix $L\in \mathbb{K}^{n\times n}$ and an $H$-normal $H$-neutral involutory matrix $X \in \mathbb{K}^{n\times n} $.  An $(H\Phi, H)$-unitary matrix  $W$ can be factorized  into a product $XL$ of an $H$-unitary  matrix $L\in \mathbb{K}^{n\times n}$ and an $H$-normal $H$-neutral involutory matrix $X\in \mathbb{K}^{n\times n} $.

\begin{thm}
\label{thm:XHX=Sigma}
Let $H \in \mathbb{K}^{n\times n}$ be a nonsingular selfadjoint matrix.
Let $\Phi \in \mathbb{K}^{n \times n}$ be an $H$-selfadjoint  involutory matrix that its negative eigenspace is hyperbolic of dimension $2m$.  Then there exists an $H$-normal $H$-neutral involutory matrix $X \in \mathbb{K}^{n\times n}$  of neutral index $m$ such that $X^{[H]}X=\Phi$.
\end{thm}

\begin{proof} $\Phi$ is involutory,  so the  Jordan form $J_\Phi$  of  $\Phi$ is a diagonal matrix  with entries $+1$ and/or  $-1$.   Since  $\Phi$ is $H$-selfadjoint, by Theorem~\ref{thm:canonical_C} or~\ref{thm:canonical_R},  the pair $(\Phi, H)$ is unitarily  similar to a canonical  pair $(J_\Phi, M)$ through some suitable  transformation matrix $Q\in \mathbb{K}^{n \times n}$.   The matrix $M$ has the same block structure as $J_\Phi$,  that is,  each block  in $M$ is size $1\times 1$.  Since  the  negative eigenspace of $\Phi$ is hyperbolic of dimension $2m$,  so  there are $2m$ blocks of $-1$ in $J_\Phi$,  where $m$ blocks  are corresponding to $+1$ in $M$  and   the other $m$  blocks are corresponding to $-1$ in $M$.  Therefore,  we can write  $(J_\Phi, M)$ as follows.
\begin{align}
Q^{-1}\Phi Q=&J_\Phi=\begin{pmatrix}
-I_m &&\\ &-I_m &\\&&I_{n-2m}
\end{pmatrix},
\end{align}
and
\begin{align}
Q^\dagger HQ=M=\begin{pmatrix}
-I_m &&\\ &I_m &\\&&\eta_{p-m, q-m}
\end{pmatrix}, \label{eq:QHQ=K}
\end{align} 
with $\eta_{p-m,q-m} = I_{p-m} \oplus -I_{q-m}$.
Let
\begin{align}
X=QPQ^{-1}, \label{eq:X=QPQ}
\end{align}
where 
\begin{align}
\label{eq:P}
P=\begin{pmatrix}
&Z_m&\\Z_m &&\\&&I_{n-2m}
\end{pmatrix}.
\end{align}
We are going to show that $X^{[H]}X=\Phi$  and  $X$ is $H$-normal $H$-neutral involutory of neutral index $m$. By direct computation, $X^2=I_n$, thus $X$ is involutory. Moreover,
\begin{align}
Q^{-1}X^{[H]}Q=P^{[M]}=\begin{pmatrix}
&-Z_m&\\-Z_m &&\\&&I_{n-2m}
\end{pmatrix},
\end{align}
and it can be verified that 
\begin{align}
P^{[M]}P=PP^{[M]}=P^{[M]}+P-I_n.
\label{eq:PMP}
\end{align}
Using Equations~(\ref{eq:QHQ=K}-\ref{eq:PMP}) in a direct calculation of $X^{[H]}X$, $XX^{[H]}$ and $X^{[H]}+X-I_n$ gives
\begin{align}
X^{[H]}X=XX^{[H]}=X^{[H]}+X-I_n=\Phi.
\end{align}
By Theorem~\ref{thm:X_XH}, $X$ is $H$-normal $H$-neutral involutory. 

Finally, diagonalizing~(\ref{eq:P}),  $X$ has $m$ negative eigenvalues $-1$ and $n-m$ positive eigenvalues $+1$. Thus $X$ is $H$-normal $H$-neutral involutory of neutral index $m$. 
\end{proof}

The matrix $X$ in Theorem~\ref{thm:XHX=Sigma} is not unique.  It is clear that  if $X$ is an $H$-normal $H$-neutral involutory solution, then $X^{[H]}$ is also a solution.  
More general,  all  the $H$-normal $H$-neutral involutory solutions of $X^{[H]}X=\Phi$ are $H$-unitarily similar to each other as  proved in  the following statements. 

\begin{lem}
\label{thm:index_m}
Let $H \in \mathbb{K}^{n\times n}$ be a nonsingular selfadjoint matrix.  Let  $\Phi \in \mathbb{K}^{n \times n}$ be an $H$-selfadjoint  involutory matrix that  its negative eigenspace is hyperbolic of dimension $2m$.   Then all the  $H$-normal $H$-neutral involutory  solutions of $X^{[H]}X=\Phi$ have neutral index $m$.  
\end{lem}

\begin{proof}
Let $X$ be a solution.   Assuming the neutral index of $X$ is $m^\prime$,  by Proposition~\ref{thm:Sigma_2m},  the negative eigenspace of $\Phi=X^{[H]}X$ has dimension $2m^\prime$, thus $2m^\prime=2m$ and so $m^\prime=m$.
\end{proof}

\begin{thm}
Let $H \in \mathbb{K}^{n\times n}$ be a nonsingular selfadjoint matrix and $\Phi  \in \mathbb{K}^{n\times n}$ be  an $H$-selfadjoint  involutory matrix   that  its negative eigenspace is hyperbolic.    Let $X_1 \in \mathbb{K}^{n\times n}$ be an $H$-normal $H$-neutral involutory solution of $X^{[H]}X=\Phi$. 
Then  $X_2$ is  an $H$-normal $H$-neutral involutory solution of  $X^{[H]}X=\Phi$ if and only if  $X_2$ is  $H$-unitarily similar to $X_1$, where the  similarity matrix commutes with $\Phi$. 
\end{thm}

\begin{proof} 
Sufficiency:  Let $X_2=LX_1L^{-1}$, where $L^{[H]}L=I_n$ and $L\Phi= \Phi L$.
Since $X_2$ is  $H$-unitarily similar to $X_1$, then by Proposition~\ref{thm:lemma1},  $X_2$ is $H$-normal $H$-neutral involutory.  Furthermore, 
\begin{align}
X_2^{[H]}X_2=&(LX_1^{[H]}L^{-1})(LX_1L^{-1})=L\Phi L^{-1}=\Phi.
\end{align}

Necessity: By Lemma~\ref{thm:index_m},  $X_1$ and $X_2$ have the same  neutral index.  By Proposition~\ref{thm:lemma2},   $X_1$ and $X_2$ are $H$-unitarily similar, that is, there exists an $H$-unitary matrix  $L$ such that  $X_2=LX_1L^{-1}$.  Since  $\Phi=X_2^{[H]}X_2=LX_1^{[H]}L^{-1}LX_1L^{-1}=L\Phi L^{-1}$,  thus $L$ commutes with $\Phi$.
\end{proof}

Now we present the decompositions of  $(H, H\Phi)$-unitary  and $(H\Phi, H)$-unitary matrices. 
\begin{thm}
\label{thm:W=LX}
Let $H \in \mathbb{K}^{n\times n}$ be a nonsingular selfadjoint matrix and  $\Phi \in \mathbb{K}^{n\times n}$ be an $H$-selfadjoint involutory matrix that  its negative eigenspace is hyperbolic.   Any $(H, H\Phi)$-unitary matrix $W \in \mathbb{K}^{n\times n}$ can be factorized (non-uniquely) as 
\begin{align}
W=LX,
\end{align}
where $L \in \mathbb{K}^{n\times n}$ is $H$-unitary  and $X \in \mathbb{K}^{n\times n}$ is $H$-normal $H$-neutral involutory.
\end{thm}
 
\begin{proof}
Let $X$ be an $H$-normal $H$-neutral involutory solution of $X^{[H]}X=\Phi$  as obtained in Theorem~\ref{thm:XHX=Sigma}.   Let $L=WX^{-1}$.  Using $W^{[H]}W=X^{[H]}X=\Phi$,  we compute  $L^{[H]}L=X^{-[H]}W^{[H]}WX^{-1}=X^{-[H]}\Phi X^{-1}=I_n$.  Thus $L$ is $H$-unitary,  $X$ is $H$-normal $H$-neutral involutory and $W=LX$.   
\end{proof}

\begin{thm}
\label{thm:W=XL}Let $H \in \mathbb{K}^{n\times n}$ be a nonsingular selfadjoint matrix and  $\Phi \in \mathbb{K}^{n\times n}$ be an $H$-selfadjoint involutory matrix that  its negative eigenspace is hyperbolic.  Any $(H\Phi, H)$-unitary matrix $W \in \mathbb{K}^{n\times n}$ can be factorized (non-uniquely) as 
\begin{align}
W=XL,
\end{align}
where $L \in \mathbb{K}^{n\times n}$ is $H$-unitary  and $X \in \mathbb{K}^{n\times n}$ is $H$-normal $H$-neutral involutory.
\end{thm}

\begin{proof}
Let $X$ be an $H$-normal $H$-neutral involutory solution of $X X^{[H]}=\Phi$  as obtained in Theorem~\ref{thm:XHX=Sigma}. Let $L=X^{-1}W$. Using $WW^{[H]}=X X^{[H]}=\Phi$,  we compute $L L^{[H]}=X^{ -1}W W^{[H]}X^{-[H]}=X^{-1}\Phi X^{-[H]}=I_n$.  Thus $L$ is $H$-unitary,  $X$ is $H$-normal $H$-neutral involutory and $W=X L$.   
\end{proof}

\section{$F=LXS= X_1L_1 S= S^\prime X^\prime L^\prime=S^\prime  L^\prime_1 X^\prime_1$ decompositions}
\label{sec:F=LXS}

In Theorems~\ref{thm:F=WS} and \ref{thm:F=SW},  we quote the results we proved in~\cite{Sui2015}.

\begin{thm}
\label{thm:F=WS}
Given a nonsingular scalar product defined by  $N\in \mathbb{K}^{n\times n}$ and a generalized sign function $\sigma: \mathbb{K}^{n\times n} \to \mathbb{K}^{n\times n}$,  a matrix $F\in \mathbb{K}^{n\times n}$ has a  decomposition 
\begin{align}
\label{eq:F=WS}
F = W S,
\end{align}
where, with $\Sigma=\sigma(F^{[N]}F)$, the matrix  $W\in \mathbb{K}^{n\times n}$ is $(N, N\Sigma^{-1})$-unitary with $(W^{[N]})^{[N]}=W$ and  the matrix $S \in \mathbb{K}^{n\times n}$ is $r$-positive-definite, $N$-selfadjoint and $N\Sigma$-selfadjoint, if and only if $F$ is nonsingular and $(F^{[N]})^{[N]}=F$.  When such a decomposition  exists it is unique, with $S$ given by $S=(\Sigma F^{[N]}F)^{1/2}$ and $W=FS^{-1}$.
\end{thm}

Similarly,  there exists a  unique left decomposition. 
\begin{thm}
\label{thm:F=SW}
If  a nonsingular matrix $F\in \mathbb{K}^{n\times n}$ has a decomposition  $F=WS$ in Theorem~\ref{thm:F=WS},  then $F$  also has a decomposition  
\begin{align}
\label{eq:F=SW}
F=S^\prime \, W,
\end{align}
where, with $\Sigma^\prime= \sigma(FF^{[N]})$,  the matrix $S^\prime \in \mathbb{K}^{n\times n}$ is  $r$-positive-definite, $N$-selfadjoint and $N\Sigma^\prime$-selfadjoint,  and the matrix $W\in \mathbb{K}^{n\times n}$  is $(N\Sigma^\prime,N)$-unitary.  When such a decomposition  exists it is unique, with $S^\prime$ given by $S^\prime=(\Sigma^\prime FF^{[N]})^{1/2}$ and $W=S^{\prime-1}F$.  $W$ in~(\ref{eq:F=WS}) and~(\ref{eq:F=SW}) is the same one. 
\end{thm}

In  Theorems~\ref{thm:F=WS} and \ref{thm:F=SW},  the generalized sign function $\sigma$ is defined in~\cite{Sui2015}.    Here we  specify that the generalized sign function $\sigma$ is the sign function defined in Section~\ref{sec:sign_function}.  Moreover, we  specify that  the matrix $H\in \mathbb{K}^{n\times n}$ is a nonsingular selfadjoint matrix defining an indefinite inner product. Then one has  the following theorem.

\begin{thm}
\label{thm:F=WS_H}
Let $H \in \mathbb{K}^{n\times n}$ be a nonsingular selfadjoint matrix. Any nonsingular  matrix $F \in \mathbb{K}^{n \times n}$ can be factorized uniquely as
\begin{align}
F = W S=S^\prime W,
\end{align}
where,  with $\Sigma= \Sign(F^{[H]}F)$ and $\Sigma^\prime= \Sign(FF^{[H]})$, 
\begin{align}
S=(\Sigma F^{[H]}F)^{1/2}\in \mathbb{K}^{n \times n}
\end{align}
is $r$-positive-definite $H$-selfadjoint and $H\Sigma$-selfadjoint, 
\begin{align}
S^\prime=(\Sigma^\prime FF^{[H]})^{1/2} \in \mathbb{K}^{n\times n}
\end{align}
 is  $r$-positive-definite $H$-selfadjoint and $H\Sigma^\prime$-selfadjoint,   and 
\begin{align}
 W=FS^{-1} = S^{\prime-1} F\in \mathbb{K}^{n \times n}
\end{align}
 is $(H, H\Sigma)$-unitary and $(H\Sigma^\prime, H)$-unitary.
\end{thm}

By Theorem~\ref{thm:Sigma}, the negative subspaces of  $\Sigma$ and $\Sigma^\prime$ are hyperbolic.  We now prove the following factorization of a nonsingular square matrix into $H$-normal matrices.

\begin{thm} 
\label{thm:F=LXS}
Let $H \in \mathbb{K}^{n\times n}$ be a nonsingular selfadjoint matrix.   Then any nonsingular  matrix $F \in \mathbb{K}^{n \times n}$ can be factorized (non-uniquely) as
\begin{align}
\label{eq:F=LXS}
F = L X S= S' L_1 X_1= S' X' L' =X'_1 L'_1 S,
\end{align}
where, with $\Sigma= \Sign(F^{[H]}F)$ and $\Sigma^\prime= \Sign(FF^{[H]})$,  
\begin{align}
S=(\Sigma F^{[H]}F)^{1/2}\in \mathbb{K}^{n \times n}
\end{align}
is $r$-positive-definite $H$-selfadjoint and $H\Sigma$-selfadjoint, 
\begin{align}
S^\prime=(\Sigma^\prime FF^{[H]})^{1/2} \in \mathbb{K}^{n\times n}
\end{align}
is  $r$-positive-definite $H$-selfadjoint and $H\Sigma^\prime$-selfadjoint,  the matrices 
$L$, $L_1$, $L^\prime$, $L^\prime_1$  $\in \mathbb{K}^{n \times n}$ are $H$-unitary,  the matrices $X$, $X_1$, $X^\prime$, $X^\prime_1$ $\in \mathbb{K}^{n \times n}$ are $H$-normal $H$-neutral involutory,   and   
\begin{align}
LX=L_1X_1= X^\prime L^\prime=X_1^\prime  L_1^\prime.
\end{align} 
\end{thm}

\begin{proof}
By  Theorem~\ref{thm:F=WS}, one has $F=WS=S^\prime W$, where $W$  is $(H, H\Sigma)$-unitary  and  $(H\Sigma^\prime, H)$-unitary with  $\Sigma=\Sign(F^{[H]}F)$ and $\Sigma^\prime=\Sign(FF^{[H]})$.  Through the definition of Sign function, $\Sigma$ and $\Sigma^\prime$ are $H$-selfadjoint.  Moreover,  by Theorem~\ref{thm:Sigma},  $\Sigma$ and $\Sigma^\prime$ satisfy the condition  in Theorem~\ref{thm:W=LX} and Theorem~\ref{thm:W=XL},  so $W=LX=X^\prime L^\prime$, where $L,L^\prime$ are $H$-unitary and $X,X^\prime$ are $H$-normal $H$-neutral involutory. 
Therefore, taking into account that $L$, $X$, $X^\prime$ and $L^\prime$ are not unique, one has $F = L X S= S' L_1 X_1= S' X' L' =X'_1 L'_1 S$, where $LX=L_1X_1= X^\prime L^\prime=X_1^\prime  L_1^\prime$.
\end{proof}

Notice that in~(\ref{eq:F=LXS}),  $S$  and $S^\prime$  are  unique,  but
$L$, $X$, $L_1$, $X_1$, $X^\prime$, $L^\prime$, $X_1^\prime$ and $L_1^\prime$ are not unique, because the factorizations $W=LX=X^\prime L^\prime$ are not unique.   In~(\ref{eq:F=LXS}),  one can choose $L=L_1$, $X=X_1$ and   $L^\prime=L_1^\prime$,  $X^\prime=X_1^\prime$.



\section*{Acknowledgements}

This research was supported by the University of Utah.






\end{document}